\documentclass[12pt,a4paper,twoside]{article}
\usepackage{amsmath,amsthm,amssymb,amsfonts,amsbsy}

\usepackage{color, colortbl} 
\usepackage{float} 
\usepackage[margin=2.5cm]{geometry} 
\usepackage{cite}
\usepackage{authblk}

\newtheorem{thm}{Theorem}[section]
\newtheorem{cor}[thm]{Corollary}
\newtheorem{lem}[thm]{Lemma}

\newtheorem{rem}[thm]{\bf{Remark}}

\numberwithin{equation}{section}
\usepackage[utf8]{inputenc}
\usepackage{pstricks-add}

\newcommand{\beq}{\begin{eqnarray}}
\newcommand{\eeq}{\end{eqnarray}}

\newcommand{\beqs}{\begin{eqnarray*}}
\newcommand{\eeqs}{\end{eqnarray*}}

\allowdisplaybreaks

\usepackage{mathtools}

\title{\bf  First and Second Maximum of Randi\'{c}  Index Among all $k-$Cyclic Graphs of a Given Order}

\author{ \bf Ali Reza Ashrafi\thanks{Corresponding author (ashrafi@kashanu.ac.ir)},  Ali Ghalavand  and Marzieh Pourbabaee
}

\affil{ \normalsize
    { \it Department of Pure Mathematics, Faculty of Mathematical Sciences, University of Kashan, Kashan 87317--53153, I. R. Iran}}

\begin{document}

\maketitle

\begin{abstract}
Suppose $G$ is a simple graph with  edge set $E(G)$. The Randi\'{c} index $R(G)$ is defined as $R(G)=\sum_{uv\in E(G)}\frac{1}{\sqrt{deg_{G}(u)deg_{G}(v)}}$, where $deg_G(u)$ denotes the vertex degree of  $u$ in $G$. In this paper, the first and second maximum of Randi\'{c}  index  among all $n-$vertex  $k-$cyclic graphs were computed.

\vskip 3mm


\vskip 3mm


\end{abstract}

\bigskip

\section{\bf Definitions and Notations}

In this section, we first describe some mathematical notions that will be kept throughout. All graphs considered in this paper are assumed to be  simple, undirected and finite without multiple edges. The undefined terms and notations are from \cite{g2,l}.

The degree of a vertex $v$ in $G$ is denoted by  $deg_G{(v)}$ and $N[v,G]$ is the set of all vertices adjacent to $v$. The notations $\Delta=\Delta(G)$ and $n_i = n_i(G)$ are  used for the maximum degree  and the number of vertices of degree $i$ in $G$, respectively.  The number of edges connecting a vertex of degree $i$ with a vertex of degree $j$ in $G$ is denoted by $m_{i,j}(G)$. An connected $n-$vertex graph $G$ is called to be $c$-cyclic if it has $n + c - 1$ edges. The number $c = c(G)$ is said to be the cyclomatic number of $G$.

Suppose $W$ is a non-empty subset of vertices in a graph $G$. The subgraph of $G$ obtained by deleting the vertices of $W$ is denoted by $G - W$ and similarly, if $F \subseteq E(G)$, then the subgraph obtained by deleting all edges in $F$ is denoted by $G - F$. In the case that $W = \{ v \}$ or $F = \{ xy \}$, the subgraphs $G-W$ and $G-F$ will shortly be written as $G - v$ or $G - xy$, respectively. Furthermore, if $x$ and $y$ are nonadjacent vertices in $G$, then the notation $G + xy$ is used for  the graph obtained from $G$ by adding an edge $xy$.

The Randi\'{c} index of a graph $G$ is defined as $R(G)=\sum_{uv\in E(G)}\frac{1}{\sqrt{deg_{G}(u)deg_{G}(v)}}$ \cite{r}. The most important mathematical properties of this number were presented in \cite{g2,l}. In the following, we first briefly review the literature on ordering graphs with Randi\'{c} index.

In \cite{de1,de2}, the first and second maximum of Randi\'c index in the class of all $n-$vertex $c$-cyclic graphs, $c = 3, 4$, were obtained.  Shiu and Zhang \cite{s} obtained the maximum value of Randi\'c index  in the class of all $n-$vertex chemical trees with $k$ pendents such that  $n < 3k - 2$.  Shi \cite{s2}  obtained some interesting results for chemical trees with respect to two generalizations of Randi\'c index. For related results we refer to the survey article  of Li and Shi \cite{l2} on the topic of Randi\'c index.

Deng et al. \cite{de}  considered various degree mean rates of an edge and gave some tight bounds for the variation of the Randi\'c index of a graph $G$ in terms of its maximum and minimum degree mean rates over its edges.

\section{\bf Five Graph Transformations}

In this section five graph transformations will be presented which are useful in computing Randi\'c index of graphs. The Transformations I and II were introduced  in \cite{ga1}.

\begin{enumerate}
\item \textbf{Transformation I.} Suppose that $G$ is a graph with a given vertex $w$ such that $deg_{G}(w)\geq1$. In addition, we assume that $P:=v_1v_2\ldots v_k$ and $Q:=u_1u_2\ldots u_l$ are two paths of lengths $k$ and $l$, respectively. Let $G_1$ be the graph obtained from $G$, $P$ and $Q$  by attaching edges $v_1w$ and $wu_1$. Define $G_2=G_1-v_1w +u_kv_1$.

\item \textbf{Transformation II.} Suppose that $G$ is a graph with given vertices $x$ and $y$ such that $deg_{G}(x),deg_{G}(y)\geq 2$ and for all $v\in N[x,G]$, $deg_{G}(x)\geq2$. In addition, we assume that $P:=v_1v_2...v_l$ and $Q:=u_1u_2...u_k$ are two paths of lengths $l$ and $k$, respectively. Define $G_1$ to be the graph obtained from $G$, $P$ and $Q$  by attaching vertices $yv_1$, $u_1x$, and $G_2=G_1-u_1x +v_lu_1$.

\item \textbf{Transformation III.}  Suppose that $G$ is a graph with  vertices $x, y, w$ and $z$ such that $\{xy,wz\}\subseteq E(G)$.  In addition, we assume that $G^\prime$ is a  trivial graph with vertex set $\{ v\}$. Define $G_1=G-\{xy\}+\{xv,vy\}$ and $G_2=G-\{wz\}+\{wv,vz\}$.

\item \textbf{Transformation IV.}  Suppose that $G$ is a graph with vertices $v_1, v_2, v_3, v_4$ and $u_1$ such that $\{v_1v_2, v_2v_3, v_4u_1\}\subseteq E(G)$, $deg_G(v_1)\geq2$,  $deg_G(v_2)=2$, $deg_G(v_3)\geq3$ and $deg_G(v_4)\geq4$. Define $G^\prime=G-\{v_4u_1\}+\{v_2u_1\}$.

\item \textbf{Transformation V.}  Suppose that $G$ is a graph with  vertices $x_1, x_2, x_3, x_4, x_5, x_6$ and $w$ such that $\{x_1x_2,x_2x_3,x_2w,x_4x_5,x_5x_6\}\subseteq E(G)$, $deg_G(x_1)=deg_G(x_2)=deg_G(x_3)=3$, $deg_G(x_4)=4$, $deg_G(x_5)=2$ and $deg_G(x_6)=1$ or $2$. Define $G^\prime=G-\{x_2w\}+\{x_5w\}$.
\end{enumerate}

It is well-known that if  the derivative $f^\prime(x)$ of a continuous function $f(x)$ satisfies $f^\prime(x)>0$ on an open interval $(a,b)$, then $f(x)$ is increasing on $(a,b)$.

\begin{lem} \label{th0} The following hold:
\begin{enumerate}
\item Let $G_1$ and $G_2$ be two graphs satisfying the conditions of Transformation $I$. Then  $R(G_2) > R(G_1)$.

\item Let $G_1$ and $G_2$ be two graphs satisfying the conditions of Transformation $II$. Then $R(G_2)>R(G_1)$.

\item Let $G_1$ and $G_2$ be two graphs as shown in Transformation $III$.
\begin{enumerate}
\item If $deg_G(x),deg_G(y)\geq3$ and $deg_G(z)\in\{1,2\}$, then $R(G_2)>R(G_1)$.

\item If $deg_G(x)=2$,  $deg_{G}(w)\geq3$ and $deg_G(z)=1$, then $R(G_2)>R(G_1)$.

\item If $deg_G(x)=2$ and $deg_G(w)=2$, then $R(G_2)=R(G_1)$.
\end{enumerate}

\item Let $G$ and $G^\prime$ be two graphs satisfying the conditions of Transformation $IV$, $\Delta(G)=q$, $N[v_4, G] = \{u_1,\ldots,u_q\}$ and $deg_G(u_i) = d_i$, $1 \leq i \leq q$. If $d_1,d_2\leq3$, then  $R(G^\prime)\geq R(G)$, with equality if and only if $deg_G(v_1)=2$,  $deg_G(v_3)=d_1=d_2=3$ and $q=d_3=d_4=4$.

\item Let $G$ and $G^\prime$ be two graphs satisfying the conditions of Transformation $V$. Then  $R(G^\prime)< R(G)$.
\end{enumerate}

\end{lem}

\begin{proof}
\begin{enumerate}
\item Let $deg_{G}(w)=x$ and $k,l\geq2$. Then by the proof of \cite[Lemma 3.1]{ga1},
\begin{eqnarray}\label{eq1}
R(G_2)-R(G_1)&>&(\frac{1}{\sqrt{2(x+1)}}+1)-(\frac{2}{\sqrt{2(x+2)}}+\frac{1}{\sqrt{2}}).
\end{eqnarray}
Let $f(x):=\frac{1}{\sqrt{2(x+1)}}-\frac{2}{\sqrt{2(x+2)}}$, for $x\in (0.8,\infty)$. Since $f(x)$ is continuous on the open interval  $(0.8,\infty)$ and $f^\prime(x)>0$ on this interval,  $f$ is increasing on  $(0.8,\infty)$.
Therefore, by Equation \ref{eq1}, $R(G_2)-R(G_1)>0.01$. The proof of other cases of $k$ and $l$ are similar and we omit them.

\item Let $deg_{G}(x)=s$ and  $k,l\geq2$. Then by the proof of \cite[Lemma 3.4]{ga1},
\begin{eqnarray}
R(G_2)-R(G_1)&>&(1+\frac{s}{\sqrt{2s}})-(\frac{s+1}{\sqrt{2(s+1)}}+\frac{1}{\sqrt{2}})\nonumber\\
&=&1-\frac{1}{\sqrt{2}}+\frac{1}{\sqrt{2}}(\sqrt{s}-\sqrt{s+1})\label{eq2}.
\end{eqnarray}
Let $g(x):=\sqrt{x}-\sqrt{x+1}$, for $x\in (0,\infty)$. Then  $g$ is increasing on  $(0,\infty)$ and so by Equation \ref{eq2}, $R(G_2)-R(G_1)\geq6\times10^{-10}$. The proof of other cases of $k$ and $l$ are similar and we omit them.

\item Suppose $deg_{G}(x)=s$, $deg_{G}(y)=r$, $deg_{G}(w)=l$ and  $deg_{G}(z)=t$. To prove $(a)$, we note that
\begin{eqnarray*}
R(G_2)-R(G_1)&=&(\frac{1}{\sqrt{sr}}+\frac{1}{\sqrt{2l}}+\frac{1}{\sqrt{2t}})-(\frac{1}{\sqrt{2s}}+\frac{1}{\sqrt{2r}}+\frac{1}{\sqrt{lt}})\\
&=&(\frac{1}{\sqrt{sr}}-\frac{1}{\sqrt{2s}}-\frac{1}{\sqrt{2r}})+(\frac{1}{\sqrt{2l}}-\frac{1}{\sqrt{2t}}-\frac{1}{\sqrt{lt}})\\
&>&-0.48+0.50>0.
\end{eqnarray*}
To prove $(b)$, we first calculate the difference between $R(G_2)$ and $R(G_1)$.
\begin{eqnarray}
R(G_2)-R(G_1)&=&(\frac{1}{\sqrt{2}}+\frac{1}{\sqrt{2l}})-(\frac{1}{2}+\frac{1}{\sqrt{l}}). \label{eq3}
\end{eqnarray}
Let $h(x)=\frac{1}{\sqrt{2x}}-\frac{1}{\sqrt{x}}$, for $x\in (0,\infty)$. Then again $h$ is increasing on  $(0,\infty)$ and hence  by Equation \ref{eq3}, $R(G_2)-R(G_1)>0.038$. For the proof of $(c)$, it is enough to notice that  $m_{i,j}(G_1)=m_{i,j}(G_2)$, $1\leq i\leq j\leq n-1$. Thus $R(G_2)=R(G_1)$, as desired.

\item Suppose that $deg_G(v_1) =s$, $deg_G(v_3) = r$ and $deg_G(v_4) = q$. Then,
\begin{eqnarray}
R(G^\prime)-R(G)&=&\left(\frac{1}{\sqrt{3s}}+\frac{1}{\sqrt{3r}}+\frac{1}{\sqrt{3d_1}}+
\frac{1}{\sqrt{(q-1)d_2}}+\sum_{i=3}^q\frac{1}{\sqrt{(q-1)d_i}}\right)\nonumber\\
&-&\left(\frac{1}{\sqrt{2s}}+\frac{1}{\sqrt{2r}}+\frac{1}{\sqrt{qd_1}}+\frac{1}{\sqrt{qd_2}}+
\sum_{i=3}^q\frac{1}{\sqrt{qd_i}}\right)\nonumber\\
&=&\left[\frac{1}{\sqrt{3s}}-\frac{1}{\sqrt{2s}}\right]+\left[\frac{1}{\sqrt{3r}}-\frac{1}{\sqrt{2r}}\right]-
\left[\frac{1}{\sqrt{qd_1}}-\frac{1}{\sqrt{3d_1}}\right]\label{equ4}\\
&-&\left[\frac{1}{\sqrt{qd_2}}-\frac{1}{\sqrt{(q-1)d_2}}\right]-
\sum_{i=3}^q\left[\frac{1}{\sqrt{qd_i}}-\frac{1}{\sqrt{(q-1)d_i}}\right].\nonumber
\end{eqnarray}
Let $f(x):=\frac{1}{\sqrt{ax}}-\frac{1}{\sqrt{bx}}$,  $x\in (0,\infty)$ and $1\leq b< a$. Then $f$ is increasing on $(0,\infty)$ and therefore by Equation \ref{equ4},
\begin{eqnarray*}
R(G^\prime)-R(G)&\geq&[\frac{1}{\sqrt{6}}-\frac{1}{\sqrt{4}}]+[\frac{1}{\sqrt{9}}-\frac{1}{\sqrt{6}}]-[\frac{1}{\sqrt{3q}}-\frac{1}{\sqrt{9}}]\\
&-&[\frac{1}{\sqrt{3q}}-\frac{1}{\sqrt{3(q-1)}}]
-[\frac{q-2}{\sqrt{q^2}}-\frac{q-2}{\sqrt{(q-1)q}}]\\
&=&\frac{1}{{6}}-\frac{2}{{3}}\frac{\sqrt{3}}{\sqrt{q}}+\frac{1}{\sqrt{3q-3}}-\frac{q-2}{\sqrt{q^2}}+\frac{q-2}{\sqrt{(q-1)q}}\geq0,
\end{eqnarray*}
with equality if and only if $deg_G(v_1)=2$,  $deg_G(v_3)=d_1=d_2=3$ and $q=d_3=d_4=4$.

\item Suppose that $deg_G(x_6) =r$. Then by definition,
\begin{eqnarray*}
R(G^\prime)-R(G)&=&(\frac{2}{\sqrt{9}}+\frac{1}{\sqrt{8}}+\frac{1}{\sqrt{2r}})-
(\frac{2}{\sqrt{6}}+\frac{1}{\sqrt{12}}+\frac{1}{\sqrt{3r}})>0.0068.
\end{eqnarray*}
\end{enumerate}
Hence the result.
\end{proof}


\begin{lem}
{\rm (See \cite{Gh2}\rm )} If $G$ is a  connected graph with $n$ vertices and cyclomatic number $\gamma$, then
$n_{1}(G)=2-2\gamma +\sum_{i=3}^{\Delta(G)} (i-2)n_i$ and $n_{2}(G) =2\gamma+ n-2-\sum_{i=3}^{\Delta(G)}(i-1)n_i.$
\end{lem}

\begin{cor}\label{cor1}
Let $G$ be a connected graph with $n$ vertices and cyclomatic number $\gamma$.
\begin{enumerate}
\item If $\gamma=5$, then $n_{1}(G)=\sum_{i=3}^{\Delta(G)} (i-2)n_i -8$ and  $n_{2}(G) = n+8-\sum_{i=3}^{\Delta(G)}(i-1)n_i$.
\item If $\gamma=6$, then $n_{1}(G)=\sum_{i=3}^{\Delta(G)} (i-2)n_i -10$ and  $n_{2}(G) = n+10-\sum_{i=3}^{\Delta(G)}(i-1)n_i$.
\end{enumerate}
\end{cor}

\vskip 3mm

Define $\Upsilon_1(n)=\{G \mid  n_{3}=8, n_2=n-8\}$, $\Upsilon_2(n) = \{G \mid n_{1}=1, n_{3}=9, n_2=n-10\}$, $\Upsilon_3(n)$ $=$ $\{G \mid n_{3}=10, n_2=n-10\}$ and $\Upsilon_4(n)$ $=$ $\{G \mid  n_{1}=1, n_{3}=11, n_2=n-12\}$.

\vskip 3mm

\begin{cor}\label{rem2}
Let $G$ be an $n-$vertex connected graph, $n \geq 12$, with cyclomatic number $\gamma$ and $\Delta(G)=3$.
\begin{enumerate}
\item If $n_1(G)=0$, then  $\gamma=5$  if and only if $G\in \Upsilon_1(n)$.
\item If $n_1(G)=1$, then  $\gamma=5$  if and only if $G\in \Upsilon_2(n)$.
\item  If $n_1(G)=0$, then $\gamma=6$  if and only if $G\in \Upsilon_3(n)$.
\item  If $n_1(G)=1$, then $\gamma=6$  if and only if $G\in \Upsilon_4(n)$.
\end{enumerate}
\end{cor}

\begin{proof}
The proof follows from Corollary \ref{cor1}.
\end{proof}

\begin{lem}
Let $G$ be a connected graph with $n$ vertices, $m$ edges and cyclomatic number $\gamma$.
\begin{enumerate}
\item If $n_1=0$ and $0<n_i< n$ for some $3\leq i\leq n-1$. Then $m_{i,i}(G)\leq n_i(T)-2+\gamma$.
\item If $n_1\geq1$ and $0<n_i< n$ for some $3\leq i\leq n-1$. Then $m_{i,i}(G)\leq n_i(T)-1+\gamma$.
\end{enumerate}
\end{lem}

\begin{proof}
$(1)$ Since $n_1=0$, $\gamma(G-v)\leq \gamma(G)-1$, for all $v\in V(G)$.  Now the proof follows from this fact that $m\leq n-1+\gamma$. The part $(2)$ is similar.
\end{proof}

Let $n$ be a positive integer. Define:
\begin{eqnarray*}
\Omega_1(n)&:=&\{G\in \Upsilon_1(n): m_{3,3}=11, m_{2,3}=2, m_{2,2}=n-9\},\\
\Omega_2(n)&:=&\{G\in \Upsilon_2(n): m_{3,3}=13, m_{2,3}=1, m_{1,2}=1, m_{2,2}=n-11\},\\
\Omega_3(n)&:=&\{G\in \Upsilon_3(n): m_{3,3}=14, m_{2,3}=2, m_{2,2}=n-11\},\\
\Omega_4(n)&:=&\{G\in \Upsilon_4(n): m_{3,3}=16, m_{2,3}=1, m_{1,2}=1, m_{2,2}=n-13\}.
\end{eqnarray*}
If $G_i\in \Omega_i(n)$ for $1\leq i\leq4$, then $R(G_1)=\frac{1}{2}n- \frac{5-2\sqrt{6}}{6}$, $R(G_2)= \frac{1}{2}n- \frac{7-(\sqrt{6}+3\sqrt{2})}{6}$,
$R(G_3)=\frac{1}{2}n- \frac{5-2\sqrt{6}}{6}$ and $R(G_4)= \frac{1}{2}n- \frac{7-(\sqrt{6}+3\sqrt{2})}{6}$.

\begin{thm}\label{tth1}
The following hold:
\begin{enumerate}
\item Let $G$ be a connected graph with $n\geq 9$ vertices and cyclomatic number 5. Then $R(G)\leq \frac{1}{2}n- \frac{5-2\sqrt{6}}{6}$, with equality if and only if $G\in \Omega_1(n)$.
\item Let $G$ be a connected graph with $n\geq 11$ vertices and cyclomatic number 6. Then $R(G)\leq \frac{1}{2}n- \frac{5-2\sqrt{6}}{6}$, with equality if and only if $G\in \Omega_3(n)$.
\end{enumerate}

\end{thm}
\begin{proof}
\begin{enumerate}
\item If $G\in \Omega_1(n)$, then $R(G)=\frac{1}{2}n- \frac{5-2\sqrt{6}}{6}$, and if  $n_1(G)=0$, then the result follows from Transformations III, IV and Lemma \ref{th0}(3,4). If $n_1(G)\geq 1$, then by  Transformations I, II, III, IV, we have $G^\prime\in\Omega_2(n)$  and by Lemma  \ref{th0}(1,2,3,4), $R(G)\leq R(G^\prime)\leq \frac{1}{2}n- \frac{7-(\sqrt{6}+3\sqrt{2})}{6}<\frac{1}{2}n- \frac{5-2\sqrt{6}}{6}$.

\item If $G\in \Omega_3(n)$, then $R(G)=\frac{1}{2}n- \frac{5-2\sqrt{6}}{6}$, and if $n_1(G)=0$, then Transformations III, IV and Lemma \ref{th0}(3,4) gives the result. Finally, if $n_1(G)\geq 1$, then by  Transformations I, II, III, IV, it follows that $G^\prime\in\Omega_4(n)$  and by Lemma  \ref{th0}(1,2,3,4),  $R(G)\leq R(G^\prime)\leq \frac{1}{2}n- \frac{7-(\sqrt{6}+3\sqrt{2})}{6}<\frac{1}{2}n- \frac{5-2\sqrt{6}}{6}$.
\end{enumerate}
\end{proof}

\begin{rem}
\begin{enumerate}
\item Let $G$ be a connected graph with $n= 8$ vertices and cyclomatic number 5. Then $R(G)\leq 4$, with equality if and only if $G$ is a $3$-regular graph.

\item Let $G$ be a connected graph with $n= 10$ vertices and cyclomatic number 6. Then $R(G)\leq 5$, with equality if and only if $G$ is a $3$-regular graph.
\end{enumerate}
\end{rem}


For positive numbers $k\geq3$, we define:
\begin{eqnarray*}
\Lambda_k^1(n)&:=&\{G: m_{3,3}=3k-4, m_{2,3}=2, m_{2,2}=n-(2k-1)\}\\
\Gamma_k^1(n)&:=&\{m_{3,3}=3k-2, m_{2,3}=1, m_{1,2}=1, m_{2,2}=n-(2k+1)\}.
\end{eqnarray*}
If $G_1\in \Lambda_k^1(n)$ and $G_2\in\Gamma_k^1(n)$, then $R(G_1)=\frac{1}{2}n- \frac{5-2\sqrt{6}}{6}$ and $R(G_2)=\frac{1}{2}n- \frac{7-(\sqrt{6}+3\sqrt{2})}{6}$.

A similar proof as Theorem \ref{tth1} proves the following general result:

\begin{thm}\label{basth1}
Let  $G$ be a connected graph with $n$ vertices and cyclomatic number $k$, where  $k\geq3$ is a positive ineger.
\begin{enumerate}
\item If $n\geq 2k-1$, then $R(G)\leq \frac{1}{2}n- \frac{5-2\sqrt{6}}{6}$, with equality if and only if $G\in \Lambda_k^1(n)$.
\item If $n= 2k-2$, then $R(G)\leq \frac{1}{2}n$, with equality if and only if $G$ is a $3$-regular graph.
\end{enumerate}
\end{thm}

Let $n$ be a positive number, $\Upsilon_5(n):=\{G: n_1=0,n_4=1,n_3=6,n_2=n-7\}$ and $\Upsilon_6(n):=\{G: n_1=0,n_4=1,n_3=8,n_2=n-9\}$. Define :
\begin{eqnarray*}
\Omega_5(n)&:=&\{G\in \Upsilon_3(n): m_{4,3}=4,  m_{3,3}=6, m_{2,3}=2, m_{2,2}=n-8\},\\
\Omega_6(n)&:=&\{G\in \Upsilon_1(n): m_{3,3}=10, m_{2,3}=4, m_{2,2}=n-10\},\\
\Omega_7(n)&:=&\{G\in \Upsilon_3(n): m_{4,3}=4,  m_{3,3}=9, m_{2,3}=2, m_{2,2}=n-10\},\\
\Omega_8(n)&:=&\{G\in \Upsilon_3(n): m_{3,3}=13, m_{2,3}=4, m_{2,2}=n-12\}.
\end{eqnarray*}
If $G_i\in \Omega_5(n)$ for $1\leq i\leq 4$,  then $R(G_1)=\frac{1}{2}n- \frac{6-(2\sqrt{3}+\sqrt{6})}{3}$,$R(G_2)= \frac{1}{2}n- \frac{5-2\sqrt{6}}{3}$, $R(G_3)=\frac{1}{2}n- \frac{6-(2\sqrt{3}+\sqrt{6})}{3}$ and $R(G_4)= \frac{1}{2}n- \frac{5-2\sqrt{6}}{3}$.

\begin{thm}\label{tth3}
The following hold:
\begin{enumerate}
\item Let $G$ be a connected graph with $n\geq 9$ vertices and cyclomatic number 5. If $G\not\in \Omega_1(n)$, then  $R(G)\leq \frac{1}{2}n- \frac{6-(2\sqrt{3}+\sqrt{6})}{3}$, with equality if and only if $G\in \Omega_5(n)$.
\item Let $G$ be a connected graph with $n\geq 11$ vertices and cyclomatic number 6. If $G\not\in \Omega_3(n)$, then  $R(G)\leq \frac{1}{2}n- \frac{6-(2\sqrt{3}+\sqrt{6})}{3}$, with equality if and only if $G\in \Omega_7(n)$.
\end{enumerate}

\end{thm}
\begin{proof}
\begin{enumerate}
\item  If $n_1(G)=0$ and $\Delta(G)=3$, then by Transformation III and Lemma \ref{th0}(3), $R(G)\leq \frac{1}{2}n- \frac{5-2\sqrt{6}}{3}<\frac{1}{2}n- \frac{6-(2\sqrt{3}+\sqrt{6})}{3}$.
If $n_1(G)=0$ and $\Delta(G)\geq4$, then by Transformations III, IV, V and Lemma \ref{th0}(3,4,5), $R(G)\leq \frac{1}{2}n- \frac{6-(2\sqrt{3}+\sqrt{6})}{3}$, with equality if and only if $G\in \Omega_5(n)$.
If $n_1(G)\geq 1$, then by  Transformations I, II, III, IV, and Lemma  \ref{th0}(1,2,3,4),  $R(G)\leq \frac{1}{2}n- \frac{7-(\sqrt{6}+3\sqrt{2})}{6}<\frac{1}{2}n- \frac{6-(2\sqrt{3}+\sqrt{6})}{3}$.
\item  If $n_1(G)=0$ and $\Delta(G)=3$, then by Transformation III and Lemma \ref{th0}(3), $R(G)\leq \frac{1}{2}n- \frac{5-2\sqrt{6}}{3}<\frac{1}{2}n- \frac{6-(2\sqrt{3}+\sqrt{6})}{3}$.
If $n_1(G)=0$ and $\Delta(G)\geq4$, then by Transformations III, IV, V and Lemma \ref{th0}(3,4,5), $R(G)\leq \frac{1}{2}n- \frac{6-(2\sqrt{3}+\sqrt{6})}{3}$, with equality if and only if $G\in \Omega_5(n)$.
If $n_1(G)\geq 1$, then by  Transformations I, II, III, IV, and Lemma  \ref{th0}(1,2,3,4), $R(G)\leq \frac{1}{2}n- \frac{7-(\sqrt{6}+3\sqrt{2})}{6}<\frac{1}{2}n- \frac{6-(2\sqrt{3}+\sqrt{6})}{3}$.
\end{enumerate}

\end{proof}
By a simple calculation one can easily see that  the Theorem \ref{tth3}(1) holds for $n=8$ and Theorem \ref{tth3}(2) holds for $n=10$. On the other hand,  Theorems \ref{tth1} and \ref{tth3} imply the following result:

 \begin{cor}
 The following hold:
 \begin{enumerate}
 \item Suppose $n\geq9$. The connected graphs with cyclomatic number 5 in the sets $\Omega_1(n)$ and $\Omega_5(n)$ have the first and second maximum Randi\'c index among all $n$-vertex connected graphs with cyclomatic number 5, respectively.
 \item Suppose $n\geq11$. The connected graphs with cyclomatic number 6 in the sets $\Omega_7(n)$ and $\Omega_8(n)$ have the first and second maximum Randi\'c index among all $n$-vertex connected graphs with cyclomatic number 6, respectively.
 \end{enumerate}

 \end{cor}


Suppose $k\geq4$ is a positive integer. Define:
\begin{eqnarray*}
\Lambda_k^2(n)&:=&\{G: m_{3,3}=3k-5, m_{2,3}=4, m_{2,2}=n-2k\}\\
\Gamma_k^2(n)&:=&\{ m_{4,3}=4,  m_{3,3}=3k-9, m_{2,3}=2, m_{2,2}=n-(2k-2)\}.
\end{eqnarray*}
If $H_1\in\Lambda_k^2(n)$ and $H_2\in \Gamma_k^2(n)$, then $R(H_1)= \frac{1}{2}n- \frac{5-2\sqrt{6}}{3}$ and $R(H_2)=\frac{1}{2}n- \frac{6-(2\sqrt{3}+\sqrt{6})}{3}$.

\begin{thm}\label{basth2}
Let $G$ be a connected graph with $n$ vertices and cyclomatic number $k\geq4$.
\begin{enumerate}
\item If $n\geq 2k-1$ and  $G\not\in \Lambda_k^1(n)$, then  $R(G)\leq \frac{1}{2}n- \frac{6-(2\sqrt{3}+\sqrt{6})}{3}$, with equality if and only if $G\in\Gamma_k^2(n)$.
\item If $n=2k-2$ and  $G$ is not a $3$-regular graph, then  $R(G)\leq \frac{1}{2}n- \frac{6-(2\sqrt{3}+\sqrt{6})}{3}$, with equality if and only if $G\in\Gamma_k^2(n)$.
\end{enumerate}
\end{thm}

\begin{proof}
The proof is similar to the proof of Theorem \ref{tth1}.
\end{proof}

We end this paper by the following result that follows from Theorems \ref{basth1} and \ref{basth2}.

\begin{thm}
Let $G$ be a connected graph with $n$ vertices and cyclomatic number $k\geq4$.
\begin{enumerate}
\item If $n\geq 2k-1$, then  the connected graphs with cyclomatic number $k$ in the sets  $\Lambda_k^1(n)$ and $\Gamma_k^2(n)$ have the first and second maximum Randi\'c index among all $n$-vertices connected graphs with cyclomatic number $k$, respectively.
    
\item If $n=2k-2$, then  the $3$-regular connected graphs and the connected graphs in the  set  $\Gamma_k^2(n)$ with cyclomatic number $k$ have the first and second maximum Randi\'c index among all $n$-vertices connected graphs with cyclomatic number $k$, respectively.
\end{enumerate}
\end{thm}


\bigskip

\noindent\textbf{Acknowledgement.} The research of the first author is partially supported by the university of Kashan under grant number 890190/5.

\end{document}